\documentclass[12pt]{article}
\usepackage{amsmath,amsthm}
\usepackage{amsfonts,amsmath}
\usepackage{amssymb}
\usepackage{epsf}
\usepackage{color}
\usepackage{graphicx}

\newtheorem{theorem}{Theorem}[section]
\newtheorem{lemma}[theorem]{Lemma}
\newtheorem{corollary}[theorem]{Corollary}

\theoremstyle{definition}

\renewcommand{\geq}{\geqslant}
\renewcommand{\leq}{\leqslant}

\renewcommand{\k}{k}

\begin{document}
%
\title{
Orthogonal trades in complete sets of MOLS
}

\author{
Nicholas J. Cavenagh\\
nickc@waikato.ac.nz\\
Department of Mathematics and Statistics\\
University of Waikato, 
Private Bag 3105, 
New Zealand \\
\\
Diane M. Donovan \\
dmd@maths.uq.edu.au\\
Centre for Discrete Mathematics and Computing,\\
School of Mathematics and Physics,\\
University of Queensland,
St Lucia 4072 Australia \\
\\
Fatih Demirkale \\
fatihd@yildiz.edu.tr\\
Department of Mathematics\\
Y\i ld\i z Technical University\\
Esenler, 34220,  \.{I}stanbul, Turkey\\
}

\maketitle

\begin{abstract}
Let $B_p$ be the Latin square given by the addition table for the integers modulo  an odd prime $p$. 
Here we consider the properties of Latin trades in $B_p$ which
preserve orthogonality with one of the $p-1$ MOLS given by the finite field construction.
We show that for certain choices of the orthogonal mate, there is a lower bound logarithmic in $p$ for the number of
times each symbol occurs in such a trade, with an overall lower bound of $(\log{p})^2/\log\log{p}$ for the size of such a trade.
Such trades imply the existence of orthomorphisms of the cyclic group which differ from a linear orthomorphism by a small amount. We also show that any transversal in $B_p$ hits the main diagonal either $p$ or at most $p-\log_2{p}-1$ times.
Finally, if $p\equiv 1\mod{6}$ we show the existence of Latin square  containing a $2\times 2$ subsquare which is orthogonal to $B_p$. 
\end{abstract}


\section{Introduction and Definitions}

Let $p$ be an odd prime.
Consider the ``complete'' set of $p-1$ MOLS of order $p$, constructed via the finite field of order $p$. (It is conjectured, but not yet proven, that a complete set of MOLS of order $p$ is unique up to isomorphism.)
The problem considered in this paper is how to change a ``small'' number of entries in one of these Latin squares so that it maintains orthogonality with at least
one other Latin square in the complete set of MOLS.

To this end, for each $k$, $1\leq k\leq p-1$,
define $B_p(k)$ to be the Latin square
where the entry in cell $(i,j)$ of $B_p(k)$ is given by $ki+j$,
for each $i,j\in {\mathbb Z}_p$.
(In the above and throughout this paper, arithmetic is performed modulo $p$ with residues in ${\mathbb Z}_p$ whenever the context makes this clear.)
Then it is well-known that
$${\mathcal B}_p:=\{B_p(1),B_p(2),\dots ,B_p(p-1)\}$$ is a set of $p-1$ MOLS of order $p$. For convenience we often write $B_p$ instead of $B_p(1)$. 

The Latin squares $B_7$ and $B_7(3)$ are given in Figure \ref{figg1}.
Observe that after each symbol is replaced by its subscript in $B_7$, the
Latin squares remain orthogonal.
We will refer to this change as an {\em orthogonal trade}.
We are interested in determining general properties of orthogonal trades; in particular lower bounds for the size of an orthogonal trade.

\begin{figure}
$$
\begin{array}{|c|c|c|c|c|c|c|}
\hline
0_3 & 1_4 & 2 & 3_0 & 4_1 & 5 & 6 \\
\hline
1 & 2 & 3 & 4 & 5 & 6 & 0 \\
\hline
2 & 3_6 & 4_5 & 5_3 & 6_4 & 0 & 1 \\
\hline
3_5 & 4_3 & 5_4 & 6 & 0 & 1 & 2 \\
\hline
4 & 5 & 6_0 & 0_1 & 1_6 & 2 & 3 \\
\hline
5_0 & 6_1 & 0_6 & 1_5 & 2 & 3 & 4 \\
\hline
6 & 0 & 1 & 2 & 3 & 4 & 5 \\
\hline
\end{array}
\quad
\begin{array}{|c|c|c|c|c|c|c|}
\hline
0 & 1 & 2 & 3 & 4 & 5 & 6 \\
\hline
3 & 4 & 5 & 6 & 0 & 1 & 2 \\
\hline
6 & 0 & 1 & 2 & 3 & 4 & 5 \\
\hline
2 & 3 & 4 & 5 & 6 & 0 & 1 \\
\hline
5 & 6 & 0 & 1 & 2 & 3 & 4 \\
\hline
1 & 2 & 3 & 4 & 5 & 6 & 0 \\
\hline
4 & 5 & 6 & 0 & 1 & 2 & 3 \\
\hline
\end{array}$$
\caption{An orthogonal trade in $B_7$}
\label{figg1}
\end{figure}

Considering a Latin square of order $n$ to be a set of ordered $($row, column, entry$)$
triples  (in this paper a subset of $\Bbb{Z}_n\times \Bbb{Z}_n \times\Bbb{Z}_n$), a {\em Latin trade} is a subset $T$ of a Latin square $L$ such that
there exists a partially filled-in Latin square $T'$ (called a
{\em disjoint mate} of $T$) such that
for each $(i,j,k)\in T$ (respectively, $T'$), there exists
unique $i'\neq i$, $j'\neq j$ and $k'\neq k$ such that
$(i',j,k),(i,j',k)$ and $(i,j,k')\in T'$ (respectively, $T$).
It follows that $(L\setminus T)\cup T'$ is a Latin square not equal to $L$.
In fact, Latin trades describe differences between Latin squares of the same order; see \cite{cavsurvey} for more details.

We define an {\em orthogonal trade}
(in ${\mathcal B}_p$)
of index $(\ell,k)$ to be a
 Latin trade $T\subset B_p(\ell)$ such that there exists a disjoint mate $T'$ such that $(B_p(\ell)\setminus T)\cup T'$
is orthogonal to $B_p(k)$. Thus Figure \ref{figg1} gives an example of an orthogonal trade in ${\mathcal B}_p$ of index $(1,3)$.

Using symmetries of ${\mathcal B}_p$, we may assume certain properties of an orthogonal trade therein.  
In this paper,  $k^{-1}$ is always taken to be the least non-negative integer representing the congruence class of $k^{-1}$ (mod $p$). 

\begin{lemma}
Let $T$ be an orthogonal trade in ${\mathcal B}_p$ of index $(\ell,k)$.
Then we may assume, without loss of generality, that
$\ell=1$, $k\leq k^{-1}$
and $(0,0,0)\in T$.
\label{iceice}
\end{lemma}

\begin{proof}
Let $1\leq x\leq p-1$.
The mapping $\phi:(a,b,ax+b)\rightarrow (a,b/\ell,(ax+b)/\ell)$ maps
$B_p(x)$ onto $B_p(x/\ell)$ and thus acts as a bijection on the set ${\mathcal B}_p$. We may thus assume that $\ell=1$.
Next, the mapping $\phi':(a,b,ax+b)\rightarrow (b,-a/x,(b-a)/x)$ maps
 maps $B_p(x)$ to $B_p(x^{-1})$ (again as part of a bijection on the set ${\mathcal B}_p$), fixing $B_p(1)$ and mapping $B_p(k)$ to $B_p(k^{-1})$.
We may thus assume $k\leq k^{-1}$. Finally, if $0\leq i\leq p-1$,
the map $\phi'':(a,b,ax+b)\rightarrow (a,b+i,ax+b+i)$ maps each element of ${\mathcal B}_p$ to itself, allowing us to assume
that $(0,0,0)\in T$.
\end{proof}

It is possible, of course, to consider Latin trades which preserve orthogonality within pairs of MOLS that do not necessarily belong to ${\mathcal B}_p$. The spectrum of possible sizes of such Latin trades is explored in \cite{DDSS16}.
However for the rest of the paper we assume that any orthogonal trade
is always in $B_p$ with the assumptions of the previous lemma.

\section{The theory of Latin trades in $B_p$}

In this section we give relevant known results and theory of Latin trades in $B_p$ - that is, the operation table for the integers modulo $p$, also known as the 
{\em back circulant Latin square}. Since an orthogonal trade necessarily is also a Latin trade in $B_p$, this theory will be useful in later sections. 

A {\em trade matrix} $A=[a_{ij}]$ is an $m\times m$ matrix with integer 
entries such that for all $1\leq i,j\leq m$: (1) $a_{ii}>0$; (2) $a_{ij}\leq 0$ whenever $i\neq j$ and 
(3) $\sum_{j=1}^m a_{ij}\geq 0$.  

\begin{lemma} {\rm (Lemma 7 of \cite{Cav1})}:\ 
If $A=[a_{ij}]$ is an $m\times m$ trade matrix, $\det(A)\leq \Pi_{i=1}^{m} a_{ii}$.
\label{bcc0}
\end{lemma}

The following lemmas are implied by the theory in \cite{Cav1}. 
The results therein are expressed in terms of symbols rather than rows; however statements about rows, columns and symbols are equivalent due to 
 equivalences of $B_p$. 
\begin{lemma}
Let $x_1,x_2,\dots ,x_{m},x_{m+1}$ be the non-empty rows of a Latin trade $T$ in $B_p$. 
Then there exists an $(m+1)\times (m+1)$ trade matrix $A$ such that 
$AX=B$, where $X=(x_1,x_2, \dots ,x_m,x_{m+1})^T$, $a_{ii}$ gives the number of entries in row $x_i$ of $T$ and $B$ is an $(m+1)\times 1$ vector of integers, each a multiple of $p$. 
Moreover, the row and column sums of $A$ are each equal to $0$.  
\label{bcc1}
\end{lemma}

\begin{lemma}
Let $A$ be an $m\times m$ trade matrix 
such that $\det(A)\neq 0$ and   
there exist $m\times 1$ vectors $X$ and $B$ such that $AX=B$, where each entry of $B$ is divisble by $p$ but each entry of $X$ is not divisible by $p$. 
Then $\det(A)$ is divisible by $p$.   
\label{bcc}
\end{lemma}

\begin{lemma} 
If $T$ is a Latin trade in $B_p$, then 
 $|T| \geq mp^{1/m}+2$.  
\label{bcc2}
\end{lemma}

We will also need the following corollary from the theory in $\cite{Cav1}$. 
\begin{lemma}
There does not exist a row $i$ of a trade matrix $A$ such that 
$a_{ii}=2$ and $a_{ij}=-2$ where $j\neq i$. 
\label{nointerc}
\end{lemma}

\begin{proof}
If such a row exists, Equation (1) of \cite{Cav1} becomes $2x_i\equiv 2x_j$ (mod $p$), which implies $x_i=x_j$ since $p$ is odd, a contradiction to the rows being distinct. 
\end{proof}

Those readers who refer back to the detail in paper \cite{Cav1} may notice that the step of proving that a trade matrix has a non-zero determinant is omitted. However Theorem \ref{irrr} in the next section addresses
the original oversight from that paper.

\section{Smallest orthogonal trade}

In this section we give a lower bound on the number of times each symbol occurs in an orthogonal trade (Theorem \ref{lowa}) and an overall lower bound for the size of an orthogonal trade (Theorem \ref{lowa2}).


Suppose that  $k\neq 1$ and symbol $s$ occurs in the rows in the set $R=\{r_1,r_2,\dots ,r_m\}$ of an orthogonal trade $T$ of index $(1,k)$.
Then clearly the set of columns of $T$ which include $s$ is equal to $\{s-r_1,s-r_2,\dots ,s-r_m\}$.
Let $\phi$ be the devolution on $R$ such that $s$ occurs in
  the set of cells
$$\{(r_i,s-\phi(r_i))\mid r_i\in R\}$$
in $T'$.
Observe that $(k-1)r_i+s$ occurs in cell $(r_i,s-r_i)$ of $B_p(k)$.
Thus, considering orthogonality, the set of orthogonal ordered pairs $\{(s,(k-1)r_i+s)\mid r_i\in R\}$ must be covered after $T$ is replaced by $T'$;
it follows that
\begin{eqnarray}
\{(k-1)r_i+s\mid r_i\in R\} & = & \{kr_i+s-\phi(r_i)\mid r_i\in R\}.
\label{roro}
\end{eqnarray}
We thus may define another permutation $\phi'$ on $R$ such that
$\phi'(r_i)=(kr_i-\phi(r_i))/(k-1)$ for each $r_i\in R$.
If $\phi'(r_i)=r_i$ for some $r_i\in R$, $\phi$ is not a devolution, a contradiction.
Similarly, $\phi'(r_i)\neq \phi(r_i)$ for each $r_i\in R$.
We thus obtain a linear system
of the form $A{\bf u}={\bf 0}$ (mod $p$),
where ${\bf u}=(r_1,r_2,\dots ,r_{m})^T$ and
$A$ is a square matrix of dimensions $m\times m$ with the following properties:
\begin{enumerate}
\item[{\rm (P1)}] Each entry of the main diagonal of $A$ is $k$.
\item[{\rm (P2)}] Each off-diagonal entry of $A$ is either $0$, $-1$ or $1-k$.
\item[{\rm (P3)}] The sum of each row and column of $A$ is $0$.
\end{enumerate}

In the example in Figure \ref{figg1} with $s=0$, we have $R=\{0,4,5\}$,  $\phi=(045)$, $\phi'=(054)$,
${\bf u}=(0,4,5)^T$ and 
$$A=\left[\begin{array}{ccc}
3 & -1 & -2 \\
-2 & 3 & -1 \\
-1 & -2 & 3 
\end{array}\right].$$

The following lemma is immediate. 
\begin{lemma}
Any symbol in an orthogonal trade occurs at least $3$ times. 
\label{th3}
\end{lemma}

Next, property (P3) above implies that $\det(A)=0$.
From Lemma \ref{iceice}, we may assume without loss of generality that $r_1=0$. 
Let $A'$ be the $(m-1)\times (m-1)$ matrix
obtained by deleting the first row and column of $A$
and let ${\bf u}'=(r_2,\dots ,r_{m})^T$.
Then $A'{\bf u}'={\bf 0}$, where $A'$ satisfies (P1), (P2) and the following properties:
\begin{enumerate}
\item[{\rm (P4)}] The sum of each row of $A$ is $0$ except for at least two rows which have a positive sum.
\item[{\rm (P5)}] The sum of each column of $A$ is $0$ except for at least two columns which have a positive sum.
\end{enumerate}

An $m\times m$ matrix $A=(a_{ij})$ is said to be {\em diagonally dominant} if
$$2|a_{ii}|\geq \sum_{j=1}^m |a_{ij}|$$
for each $i\in [m]$. Clearly $A'$ above is diagonally dominant.

\begin{theorem} {\rm (\cite{HoJo, Taus})}
If $A$ is diagonally dominant and irreducible and there is an integer
$k\in [m]$ such that
\begin{eqnarray}
2|a_{kk}| & > & \sum_{j=1}^m |a_{kj}|,
\label{eqqq}
\end{eqnarray}
then $A$ is non-singular.
\label{irrr}
\end{theorem}

Thus if $A'$ is irreducible, we have from the previous theorem,
det$(A')\neq 0$.
However the case when $A'$ is reducible can be dealt with in the following
lemma, which is easy to prove.
\begin{lemma}
Let $A'$ be a diagonally dominant matrix satisfying
{\rm (P1)}, {\rm (P2)}, {\rm (P4)} and {\rm (P5)} above.
Then there exists an irreducible, diagonally dominant $m'\times m'$ matrix $A''$
with $m'\leq m$ satisfying
{\rm (P1)}, {\rm (P2)} and Equation $(\ref{eqqq})$ above.
\end{lemma}

Thus there exists an $m'\times m'$  matrix $A''$, satisfying
{\rm (P1)}, {\rm (P2)} and Equation $(\ref{eqqq})$ above,
with non-zero determinant, where $m'\leq m$.
Moreover, $A''$ is a type of {\em trade matrix} as defined in
the previous section.
From Lemma \ref{bcc0},  
 the determinant of $A''$ is bounded above by $k^{m-1}$.
Thus from Lemma \ref{bcc}, $p< k^{m-1}$ and we have shown the following.
\begin{theorem}
Let $K=${\rm\ min}$\{k,k^{-1}\}$.
The number of times each symbol occurs in an orthogonal trade is greater than $\log_K{p+1}$.
\label{lowa}
\end{theorem}

 We next find a lower bound on the size of $T$.
\begin{theorem}
If $T$ is an orthogonal trade of index $(1,k)$,
then
$$|T|> \frac{\log{p}\log_K{p}}{\log\log_K{p}}.$$
where $K=${\rm\ min}$\{k,k^{-1}\}$.
\label{lowa2}
\end{theorem}

\begin{proof}
Let $T$ contain $m$ distinct symbols and let 
$s_i$ be the the number of times symbol $i$ occurs in $T$, where $1\leq i\leq m$. 
From Lemma \ref{bcc2},  
for any Latin trade in $B_p$, $\sum_{i=1}^{m} s_i=|T|> mp^{1/m}$.
Let $x=|T|/m=
(\sum_{i=1}^{m} s_i)/m$. From Lemma \ref{th3}, $x\geq 3$. 
Also, from above, $x> p^{1/m}$ which implies that $m> (\log{p})/(\log{x})$.
Thus $|T|> (x/\log{x})\log{p}$. But the function $x/\log{x}$ is strictly increasing for $x>e$; 
thus the result follows from the previous theorem.
\end{proof}

\section{Orthogonal trades permuting entire rows} \label{section:entirerows}

In this section we
investigate the case when $T$ and
$T'$ are constructed taking complete rows of $B_p$ and permuting them.
It turns out that such orthogonal trades arise from considering a symbol from an arbitrary orthogonal trade.

\begin{theorem} Let $T$ be an orthogonal trade in $B_p$.
Let $R$ be the set of rows that contain a particular symbol $s$ in $T$.
Then there exists an orthogonal trade of size $p|R|$ constructed by permuting the rows of $R$.
\label{rara}
\end{theorem}

\begin{proof} Fix $s\in {\Bbb Z}_p$.
Equation \ref{roro} implies that
$$\{((k-1)r_i+s+j)\mid r_i\in R,j\in {\mathbb Z}_p\} = \{(kr_i+s-\phi(r_i)+j)\mid
r_i\in R,j\in {\mathbb Z}_p\}.$$
Thus if we replace row $r_i$ with row $\phi(r_i)$ for each $r_i\in R$ we obtain an orthogonal trade.
\end{proof}

\begin{figure}
$$
\begin{array}{|c|c|c|c|c|c|c|}
\hline
0_4 & 1_5 & 2_6 & 3_0 & 4_1 & 5_2 & 6_3 \\
\hline
1 & 2 & 3 & 4 & 5 & 6 & 0 \\
\hline
2 & 3 & 4 & 5 & 6 & 0 & 1 \\
\hline
3 & 4 & 5 & 6 & 0 & 1 & 2 \\
\hline
4_5 & 5_6 & 6_0 & 0_1 & 1_2 & 2_3 & 3_4 \\
\hline
5_0 & 6_1 & 0_2 & 1_3 & 2_4 & 3_5 & 4_6 \\
\hline
6 & 0 & 1 & 2 & 3 & 4 & 5 \\
\hline
\end{array}
\quad
\begin{array}{|c|c|c|c|c|c|c|}
\hline
0 & 1 & 2 & 3 & 4 & 5 & 6 \\
\hline
3 & 4 & 5 & 6 & 0 & 1 & 2 \\
\hline
6 & 0 & 1 & 2 & 3 & 4 & 5 \\
\hline
2 & 3 & 4 & 5 & 6 & 0 & 1 \\
\hline
5 & 6 & 0 & 1 & 2 & 3 & 4 \\
\hline
1 & 2 & 3 & 4 & 5 & 6 & 0 \\
\hline
4 & 5 & 6 & 0 & 1 & 2 & 3 \\
\hline
\end{array}
$$
\caption{An orthogonal trade derived from Figure $\ref{figg1}$ as in Theorem $\ref{rara}$.}
\end{figure}

In fact, the existence of an orthogonal trade permuting entire rows
is equivalent to the existence of any matrix $A$ satisfying properties from Section 2.
\begin{corollary}
Let $A$ be an $m\times m$ matrix satisfying properties {\rm (P1)}, {\rm (P2)} and {\rm (P3)} from Section $2$.
Suppose furthermore there is a solution to
 $A{\bf u}={\bf 0}$ where ${\bf u}=(r_1,r_2,\dots ,r_m)^T$ and $r_1,r_2,\dots ,r_m$ are distinct residues in ${\mathbb Z}_p$.
Then there exists an orthogonal trade $T$ of index $(1,k)$ whose disjoint mate $T'$ is formed by permuting the rows $r_1,r_2,\dots ,r_m$ of $T$.
\label{striker}
\end{corollary}

\begin{proof}
Define $\phi(r_i)=r_j$ if and only if $A_{ij}=-1$ and
define $\phi'(r_i)=r_j$ if and only if $A_{ij}=-(k-1)$.
Then $\phi$ and $\phi'$ are disjoint devolutions on the set $\{r_1,r_2,\dots ,r_m\}$ and $\phi'(r_i)(k-1)=kr_i-\phi(r_i)$ for each $i$, $1\leq i\leq m$.
In turn, Equation \ref{roro} is satisfied. The proof then follows by Theorem \ref{rara}.
\end{proof}

From the previous theorem and Theorem \ref{lowa}, we have the following.

\begin{corollary}
Let $T$ be an orthogonal trade of index $(1,k)$ with disjoint mate $T'$ such that $T'$ is the permutation of $m$ entire rows of $T$.
Then $m\geq \log_K{p+1}$,
where  $K=${\rm\ min}$\{k,k^{-1}\}$.
\label{permm}
\end{corollary}

\begin{theorem}
There exists an orthogonal trade $T$ with disjoint mate $T'$ such that $T'$ is the permutation of $3$ entire rows of $T$
 if and only if $p\equiv 1\mod{6}$.
\end{theorem}

\begin{proof}
From the theory in the previous section, the determinant of $A'$ must be equal to $k^2-k+1$.
However $k^2-k+1=0$ has a solution mod $p$ if and only if $-3$ is a square mod $p$.
Elementary number theory can be used to show that $-3$ is a square mod $p$ if and only if $p\equiv 1\mod{6}$.
Finally, if $p\equiv 1\mod{6}$, replacing row $0$ with row $1$, row $1$ with row $k$ and row $k$ with row $0$ in $B_p$ creates a Latin square which remains
 orthogonal to $B_p(k)$.
\end{proof}

It is an open problem to determine whether there exists an orthogonal trade permuting a bounded number of rows for any odd prime $p$.

\section{Orthogonal trades via Latin trades in $B_p$}

Our aim in this section is to construct an orthogonal trade $T$ with index $(1,2)$ with disjoint mate $T'$ such that $T'$ permutes $O(\log{p})$ entire rows of $T$.
We do this by showing the existence of orthogonal Latin trades in $B_p$ with
size $O(\log{p})$.

\begin{theorem}
For each prime $p$ there exists a Latin trade $T$ of size $O(\log{p})$ within
$B_p$ such that each symbol occurs either twice in $T$ or not at all.
\label{szz}
\end{theorem}

\begin{theorem}
For each prime $p$ there exists an orthogonal trade of index $(1,2)$ permuting $O(\log{p})$ rows.
\label{yoyo}
\end{theorem}

\begin{proof}
From Section 2, 
 the trade matrix $A$ 
corresponding to the trade $T$ given
by the previous theorem has the following properties.
Firstly, the number of rows (and the number of columns of $A$) is $O(\log{p})$.
Secondly, each entry of the main diagonal is $2$, every other entry is either $-2$, $-1$ or $0$ and the row and column sums are at least $0$.
From Lemma \ref{nointerc}, there are no entries $-2$. 
Moreover, from Lemma \ref{bcc2}, $A{\bf u}={\bf 0}$ has a solution in ${\mathbb Z}_p$ where the entries of ${\bf u}$ are distinct.
The result follows by Corollary \ref{striker}.
\end{proof}

In order to prove Theorem \ref{szz} we modify a construction given by Szabados \cite{sz} which proved the following.

\begin{theorem} {\rm (Szabados, \cite{sz})}
For each prime $p$ there exists a Latin trade of size at most $5\log_2{p}$ within
$B_p$.
\end{theorem}

Since our proof is a modification of that given in \cite{sz} (which was in turn inspired by classic results on dissections of squares by Brooks, Smith, Stone  and Tutte \cite{BSST,tut} and Trustum \cite{tru}) we borrow from the notation given in \cite{sz}.

A {\em dissection} of order $k$ of a rectangle $R$ with integer sides is a set of $k$ squares of integral side which partition the area of the rectangle
(i.e. they cover the rectangle and overlap at most on their boundaries).
A dissection is said to be $\oplus$-free if no four of them share a common point.

For the following definition we position our rectangle $R$ with a corner at the origin, its longest side along the positive $x$-axis and
another side along the negative $y$-axis.

We say that a dissection is {\em good} if it is:
\begin{enumerate}
\item[(G1)] $\oplus$-free;
\item[(G2)] the square with the origin as a corner point has side at least $3$; 
\item[(G3)] there is no line of gradient $-1$ intersecting corner points of more than one square; and
\item[(G4)] the lines $y=1-x$ and $y=2-x$ do not intersect corner points of any square.
\end{enumerate}

We wish to construct a good dissection of a rectangle of dimensions $n\times (n+3)$ for any $n\geq 3$.
We first deal with small values of $n$.  
\begin{lemma}
There exists a good dissection of the rectangle $n\times (n+3)$ for $3\leq n\leq 14$
with at most $8$ squares.
\end{lemma}

\begin{proof}
In every case, make one of the squares an $n\times n$ square with the origin as a corner point and another $3\times 3$ subsquare with $(n+3,-n)$ as a corner point. 
Then (G2) and (G4) are satisifed. It is then easy to find a dissection of the remaining $(n-3)\times 3$ rectangle satisfying (G1) and (G3), using at most $6$ squares for each case (one can simply use a greedy algorithm, cutting off a largest possible square at each step). 
\end{proof}

Figure \ref{good5} displays a good dissection of the $5\times 8$ rectangle into $5$ squares.

\begin{figure}[ht]
\begin{center}
\includegraphics{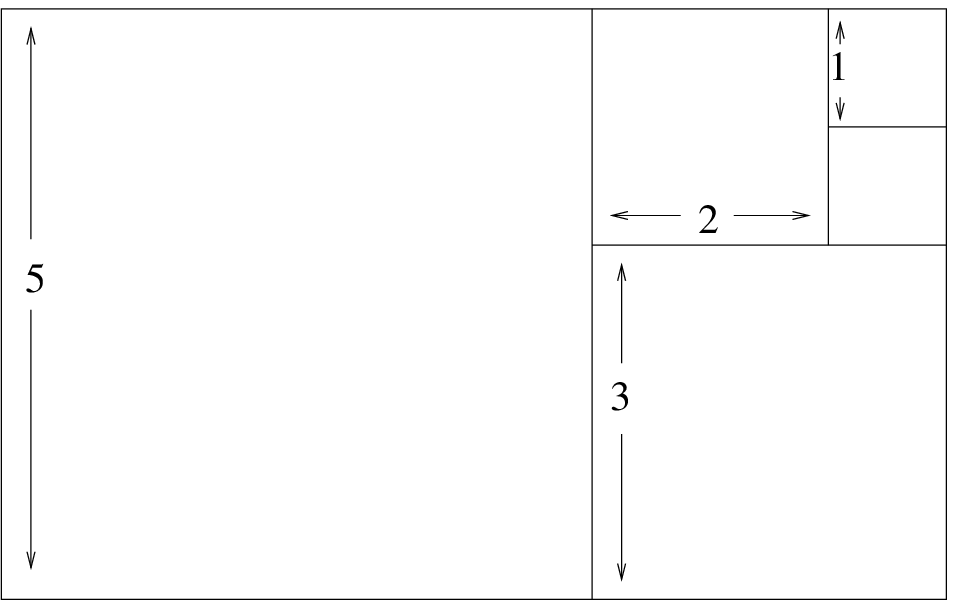}
\end{center}
\caption{An example of a good dissection into squares}
\label{good5}
\end{figure}

For $n$ of the form $4k+z$ with $k\geq 3$, $z\in \{3,4,5,6\}$
we may dissect an $n\times (n+3)$ rectangle into at most $5$ squares and a rectangle of size
$2k\times 2(k+3)$, as shown in Figure \ref{sz}.

\begin{figure}[htbp]
\begin{center}
\input{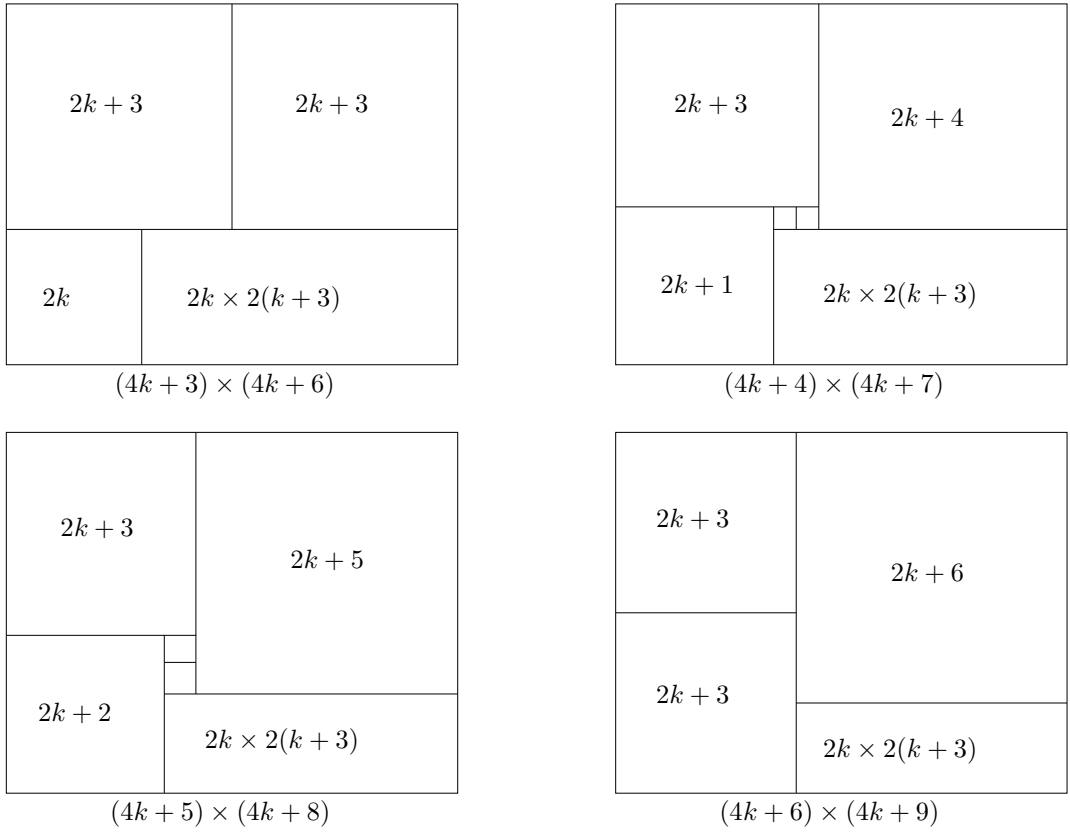}
\end{center}
\caption{
Dissecting a rectangle of size $n\times (n+3)$
\ (Figure $1$. from \cite{sz}) 
}
\label{sz}
\end{figure}

\begin{lemma}
For each $n\geq 3$, there exists a good dissection of
an $n\times (n+3)$ rectangle using at most $3+5\log_4{(n+1)}$ squares.
\label{goodies}
\end{lemma}

\begin{proof}
From the previous lemma, the result holds for $3\leq n\leq 14$.
If $n\geq 15$, write $n=4k+z$ where $k\geq 3$ and $z\in \{3,4,5,6\}$ and use a dissection as in Figure \ref{sz},
recursively using a good dissection of the $k\times (k+3)$ rectangle with the length of each square doubled.
Property (G2) of the smaller rectangle ensures that (G1) holds for the larger rectangle.
Property (G2) clearly holds for the larger rectangle as $k\geq 3>0$.
Next, property (G4) (avoiding the line $y=1-x$) for the smaller rectangle ensures that (G3) holds for the larger rectangle.
Finally, property (G4)(avoiding the line $y=2-x$) for the smaller rectangle ensures than (G4) holds for the larger rectangle. Note in the previous that $y=2-x$ with respect to the larger rectangle cannot hit any corners of squares in the smaller rectangle because each square has even length side.

Suppose such a recursion occurs $\alpha$ times to an intial rectangle of order $m\times (m+3)$ where $3\leq m\leq 14$. 
Then $n\geq g^{\alpha}(m)$, where $g(m)=4m+3$ and $g^{\alpha}$ is the function $g$ composed with itself ${\alpha}$ times. 
Observe that $g^{\alpha}(m)=4^{\alpha}m+4^{\alpha}-1$. 
Thus $(n-4^{\alpha}+1)/4^{\alpha}\geq 3$ 
and $\alpha\leq \log_4{(n+1)}-1$. Each recursive step gives at most $5$ extra squares; with at most $8$ squares in the initial step, the result follows
 by Lemma \ref{goodies}.
\end{proof}

The proof of Theorem \ref{szz} now follows from the following theorem, which is
outlined in \cite{sz} and first established in \cite{Dr}.

\begin{theorem}
Suppose there exists a good dissection of order $k$ of an $n\times m$ rectangle.
Then there exists a Latin trade $T$ in the addition table for the integers modulo $m+n$ (i.e. $B_{m+n}$ if $m+n$ is prime) such that each entry of $T$ appears exactly twice and $T$ has size $2k+2$.
\end{theorem}

\begin{proof}
The proof follows from the construction, first given in \cite{Dr}, showing that a dissection of a right-angled isoceles triangle (with two sides of length $p$) into smaller, integer-sided 
right-angled isoceles triangles 
  gives rise to a Latin trade $T$ in $B_p$, provided that no point is the vertex of $6$ of the smaller triangles.
In such a construction, the number of smaller triangles gives the size of the Latin trade. 
Reposition the triangle on the Euclidean plane so that its vertices have positions $(0,0)$, $(0,p)$ and $(p,0)$. 
Then the coordinates of the vertices of the smaller triangles give precisely the cells of $B_p$ which $T$ occupies. 


Next, reposition the $n\times m$ rectangle 
so that its vertices have coordinates $(0,0),(0,m),(n,0)$ and $(m,n)$. 
Embed this rectangle into 
 an isocoles right-angled triangle as above (with two equal sides of length $n+m$). Dissect each square in the good dissection into two triangles so that the sides of each triangle are parallel to the larger triangle. 
This gives a dissection of the right-angled triangle into $2k+2$ smaller right-angled isoceles triangles. 
Reposition the triangle as above. 

Then in our construction, the line segements of gradient $-1$ contain the same symbol in $B_p$. Each such line segment intersects only two corners of squares and thus only two vertices of triangles. Together with condition (G3), this ensures that each symbol occurs exactly twice in the Latin trade. 
\end{proof}

Apply the process in the above theorem to the example in Figure \ref{good5}. 
This results in the following Latin trade $T$ in $B_{13}$ (with a unique disjoint mate $T'$):
$$
\begin{array}{l}
T:=\{(0,0,0),(0,5,5),(5,0,5),(5,3,8),(8,0,8),(5,5,10),(7,3,10),  \\
(7,4,11),(8,3,11),(7,5,12),(8,4,12),(8,5,0)\}.  \\
T':=\{(0,0,5),(0,5,0),(5,0,8),(5,3,10),(8,0,0),(5,5,5),(7,3,11),  \\
(7,4,12),(8,3,8),(7,5,10),(8,4,11),(8,5,12)\}.  \\
\end{array} 
$$
Note that each symbol occurs twice. For a general proof of why this construction gives a Latin trade, see \cite{Dr}. 

\section{Orthomorphisms of cyclic groups and transversals in $B_p$}

As in previous sections we assume that $p$ is prime. 
An {\em orthomorphism} of the cyclic group ${\mathbb Z}_p$ is a permutation
$\phi$ of the elements of ${\mathbb Z}_p$ such that $\phi(x)-x$ is also
a permutation.
We note that orthomorphisms have a more general definition for arbitrary groups. 
However in this section we assume that orthomorphisms are of the cyclic group only. 
Trivial examples of orthomorphisms are given by $\phi(x)=kx$ for any $k$, $2\leq k\leq p-1$.
Given any orthomorphism $\phi$,
construct a Latin square $L_{\phi}$ by placing $\phi(r) +c$ in cell $(r,c)$. Then
by definition $L_{\phi}$ is orthogonal to $B_p$.

Given two orthomorphisms $\phi$ and $\phi'$, the {\em distance} between $\phi$ and $\phi'$ is defined to be the number of values $x$ for which $\phi(x)\neq \phi'(x)$.
Corollary \ref{permm} implies the following result about orthomorphisms.
\begin{theorem}
Let $\phi'$ be an orthomorphism not equal to $\phi(x)=kx$. Then the distance between $\phi$ and $\phi'$ is at least
$\log_K{p}+1$ where $K={\rm\ min}\{k,k^{-1}\}$.
\end{theorem}

A {\em transversal} of a Latin square of order $n$ is a set of ordered triples that include each row, column and symbol exactly once. Given any orthomorphism $\phi$, the set of triples $(x,\phi(x)-x,\phi(x))$ is a transversal of $B_p$.
For example, if $\phi(x)=2x$ we obtain a transversal on the main diagonal of $B_p$.
So we have the following corollary.
\begin{corollary}
Any transversal of $B_p$ not equal to the main diagonal has at least
$\log_2{p}+1$ elements off the main diagonal.
\end{corollary}

From Theorem \ref{yoyo}, we also have the following. 
\begin{theorem}
There exists a transversal of $B_p$ not equal to the main diagonal which 
has $O(\log{p})$ elements not on the main diagonal. 
\end{theorem}

\section{A construction for an orthogonal trade with size not divisible by $p$}

In this section we construct an orthogonal trade of size not divisible by $p$ whenever $p\equiv 1$ (mod $6$). 
Figure \ref{figg1} gives the construction for $p=7$. Figure \ref{figg3(p-1)trade} is an example of the construction for $p=13$, 
where the trademate is shown via subscripts.

\begin{figure}
$$\begin{array}{|c|c|c|c|c|c|c|c|c|c|c|c|c|}
\hline
0_4 & 1_5 & 2_6 & 3 & 4_0 & 5_1 & 6_2 & 7 & 8 & 9 & 10 & 11 & 12\\
\hline
1 & 2 & 3 & 4 & 5 & 6 & 7 & 8 & 9 & 10 & 11 & 12 & 0 \\
\hline
2 & 3 & 4 & 5 & 6 & 7 & 8 & 9 & 10 & 11 & 12 & 0 & 1 \\
\hline
3 & 4_8 & 5_9 & 6_7 & 7_4 & 8_5 & 9_6 & 10 & 11 & 12 & 0 & 1 & 2 \\
\hline
4_7 & 5_4 & 6_5 & 7_6 & 8 & 9 & 10 & 11 & 12 & 0 & 1 & 2 & 3 \\
\hline
5 & 6 & 7 & 8 & 9 & 10 & 11 & 12 & 0 & 1 & 2 & 3 & 4  \\
\hline
6 & 7 & 8_{12} & 9_{10} & 10_{11} & 11_{8} & 12_9 & 0 & 1 & 2 & 3 & 4 & 5 \\
\hline
7_{10} & 8_{11} & 9_8 & 10_9 & 11_7 & 12 & 0 & 1 & 2 & 3 & 4 & 5 & 6 \\
\hline
8 & 9 & 10 & 11 & 12 & 0 & 1 & 2 & 3 & 4 & 5 & 6 & 7  \\
\hline
9 & 10 & 11 & 12_0 & 0_1 & 1_2 & 2_{12} & 3 & 4 & 5 & 6 & 7 & 8 \\
\hline
10_0 & 11_1 & 12_2 & 0_{12} & 1_{10} & 2_{11} & 3 & 4 & 5 & 6 & 7 & 8 & 9 \\
\hline
11 & 12 & 0 & 1 & 2 & 3 & 4 & 5 & 6 & 7 & 8 & 9 & 10  \\
\hline
12 & 0 & 1 & 2 & 3 & 4 & 5 & 6 & 7 & 8 & 9 & 10 & 11 \\
\hline
\end{array}$$
\caption{An orthogonal trade of index $(1,4)$ and size $36$ in $B_{13}$}
\label{figg3(p-1)trade}
\end{figure}

Let $k\geq 2$; since $p\equiv 1$ (mod $6$) there exists $k$ such that
 $k^2-k+1$ is divisible by $p$
(since $-3$ is a square modulo $p$ if and only if $p-1$ is divisible by $6$).
Note that if $k$ is a solution then $1-k$ is also a solution modulo $p$; thus we assume in this section that $k$ is an integer such that $2\leq k\leq (p+1)/2$.
We remind the reader that all values are evaluated modulo $p$ with a residue between $0$ and $p-1$.
We define the following subsets $T_0,T_1,\dots ,T_{k-1}$ of $B_{p}$:

$$T_0:=\{(0,j,j),(0,\k+j,\k+j)\mid 0\leq j\leq \k-2\}$$
and if $1\leq i\leq \k-1$,
$$T_i:=\{(i(\k-1),j,i(\k-1)+j)\mid i\leq j\leq 2(\k-1)\}\cup$$
$$\{(i(\k-1)+1,j,i(\k-1)+j+1)\mid 0\leq j\leq \k+i-2\}.$$

We then define $T$ to be the union of all these sets; i.e.
$$T:=\bigcup_{i=0}^{\k-1} T_i.$$

The condition $p>2k-2$ ensures that the above sets are disjoint.
Observe that $|T_0|=2(\k-1)$ and
for each $1\leq i\leq \k-1$, $|T_i|=3\k-2$. Thus $|T|=3\k(\k-1)$ 
which is not divisible by $p$. In the case where $p=k^2-k+1$ (where $p$ and $k$ are integers), the size of $T$ is $3(p-1)$, but in general may be larger relative to $p$.   

We will show that $T$ is a Latin trade which preserves orthogonality between
the Latin squares $B_{p}$ and 
$B_p(\k)$. 


With this aim in view, we define a partial Latin square $T'$ which we
will show is a disjoint mate of $T$.
Let
\begin{eqnarray*}
T_0':=\{(0,j,\k+j),(0,\k+j,j)\mid 0\leq j\leq \k-2\}
\end{eqnarray*}
and if $1\leq i\leq \k-1$,
\begin{eqnarray*}T_i':=&&\{(i(\k-1),j,i(\k-1)+j+\k)\mid i\leq j\leq \k-2\}\cup\\
&&\{(i(\k-1),j,i(\k-1)+j+1)\mid \k-1\leq j\leq \k+i-2 \}\cup\\
&&\{(i(\k-1),j,(i-1)(\k-1)+j)\mid \k+i-1\leq j\leq 2(\k-1)\}\cup\\
&&\{(i(\k-1)+1,j,i(\k-1)+j+\k)\mid 0\leq j\leq i-1\}\cup\\
&&\{(i(\k-1)+1,j,i(\k-1)+j)\mid i\leq j\leq \k-1\}\cup\\
&&\{(i(\k-1)+1,j,i(\k-1)+j-\k+1)\mid \k\leq j\leq \k+i-2\}.
\end{eqnarray*}
Note that for $i=k-1$ the first set in $T_i'$ is empty and for 
$i=0$ the last set is empty. 
We define $T'$ to be the union of the above sets; i.e.
$$T':=\bigcup_{i=0}^{\k-1} T_i'.$$

By observation, $T$ and $T'$ occupy the same set of cells and are disjoint.
We next check that corresponding rows contain the same set of symbols.
This is easy to check for row $0$.
Let $1\leq i\leq\k-1$.
Row $i(\k-1)$ of $T'$ contains the symbols
$\{i(\k-1)+j+\k\mid i\leq j\leq \k-2\}=\{i(\k-1)+j \mid i+\k\leq j\leq 2(\k-1)\}$,
$\{i(\k-1)+j+1\mid \k-1\leq j\leq \k+i-2\}=\{i(\k-1)+j \mid \k\leq j\leq \k+i-1 \}$ and
$\{(i-1)(\k-1)+j\mid \k+i-1\leq j\leq 2(\k-1)\}=\{i(\k-1)+j \mid i\leq j\leq \k-1 \}$.
Thus row $i(\k-1)$ of $T'$ contains the same set of symbols as the corresponding row of $T$.

Next, row $i(\k-1)+1$ of $T'$ contains the symbols
$\{i(\k-1)+j+\k)\mid 0\leq j\leq i-1\}=
\{i(\k-1)+j+1)\mid \k-1\leq j\leq i+\k-2\}$,
$\{i(\k-1)+j)\mid i\leq j\leq \k-1\}=
\{i(\k-1)+j+1)\mid i-1\leq j\leq \k-2\}$ and
$\{i(\k-1)+j-\k+1)\mid \k\leq j\leq \k+i-2\}=
\{i(\k-1)+j+1)\mid 0\leq j\leq i-2\}$.
Thus  row $i(\k-1)+1$ of $T'$ contains the same set of symbols as the corresponding row of $T$.
We have shown that $T$ and $T'$ share the same sets of symbols in correpsonding rows.

We now show this property for the columns.
It suffices to show that each symbol in a column of $T'$ occurs within the same column of $T$.
First consider elements of $T_0'$. Let $0\leq j\leq \k-2$.
Then symbol $j+\k$ in cell $(0,j)$ of $T_0'$ belongs also
to cell $(\k,j)$ of $T_1$. Moreover symbol $j$ in cell $(0,\k+j)$ of $T_0'$
belongs also to cell $((\k-1)^2,\k+j)$ of $T_{\k-1}$ since
$(\k-1)^2+\k$ is divisible by $p$.

In this paragraph we deal with symbols which occur in row $i(\k-1)$ of $T_i'$ for some
$1\leq i\leq \k-1$.
Consider symbol $i(\k-1)+j+\k$ in column $j$ of $T_i'$ where
 $i\leq j\leq \k-2$.
This symbol also lies in cell $((i+1)(\k-1)+1,j)$ of $T_{i+1}$.
Consider symbol $i(\k-1)+j+1$ in column $j$ of $T_i'$ where
$\k-1\leq j\leq \k+i-2$.
This symbol lies in
cell $(i(\k-1)+1,j)$ of $T_i$.
Consider symbol $(i-1)(\k-1)+j$ in column $j$ of $T_i'$ where
$\k+i-1\leq j\leq 2(\k-1)$.
This symbol lies in cell $((i-1)(\k-1),j)$ of $T_{i-1}$.

Finally, to verify that $T'$ is indeed a disjoint mate of $T$,
we look at symbols which occur in row $i(\k-1)+1$ of $T_i'$ for some
$1\leq i\leq \k-1$.
Consider symbol
$i(\k-1)+j+\k$ which occurs in column $j$ of $T_i'$
where $0\leq j\leq i-1$.
This symbol occurs in cell
$((i+1)(\k-1)+1,j)$ of $T_{i+1}$ (if $i<\k-1$) or $T_0$ (if $i=\k-1$).
Next consider symbol
$i(\k-1)+j$ which occurs in column $j$ of
 $T_i'$ where $i\leq j\leq \k-1$.
This symbol occurs in cell
$(i(\k-1),j)$ of $T_i$.
Thirdly, consider symbol
$i(\k-1)+j-\k+1$
of $T_i'$ where
$\k\leq j\leq \k+i-2$.
This symbol occurs in cell
$((i-1)(\k-1),j)$ of $T_{i-1}$.

We have shown that $T$ is a Latin trade in $B_p$ with disjoint mate $T'$.
Next we show orthogonality.
it suffices to show that for each element $(r,c,r+c)\in T$,  there is a cell $(r',c')\in T'$ containing $r+c$ such that $(r',c')$ contains
$r\k+ c$ in $B_p(\k)$ (equivalently, $r'\k+c'=r\k+c$).

Firstly, let $(0,j,j)\in T_0$ where $0\leq j\leq k-2$. 
Then $(p-k,j+k-1,j)\in T_{k-1}'$. 
Next let $1\leq i\leq k-1$.
Let $(i(\k-1),j,i(\k-1)+j)\in T_i$ where $i\leq j\leq \k-1$ ($i\leq j\leq k-2$ when $i=0$).
Then $((i-1)(\k-1),j-1,i(\k-1)+j)\in T_{i-1}'$.
Next let $1\leq i\leq k-1$.
Let $(i(\k-1),j,i(\k-1)+j)\in T_i$ where $\k\leq j\leq \k+i-1$.
Then $(i(\k-1)+1,j-\k,i(\k-1)+j)\in T_i'$.
Let $0\leq i\leq k-2$.
Let $(i(\k-1),j,i(\k-1)+j)\in T_i$ where $\k+i\leq j\leq 2(\k-1)$.
Then $((i+1)(\k-1)+1,j-\k+1,i(\k-1)+j)\in T_{i+1}'$.

Next let $1\leq i\leq k-1$.
Let $(i(\k-1)+1,j,i(\k-1)+1+j)\in T_i$ where $i-1\leq j\leq \k-2$.
Then $(i(\k-1),j+\k,i(\k-1)+1+j)\in T_i'$.
Let $2\leq i\leq k-1$. Let $(i(\k-1)+1,j,i(\k-1)+1+j)\in T_i$ where $0\leq j\leq i-2$.
Then $((i-1)(\k-1),j+\k-1,i(\k-1)+1+j)\in T_{i-1}'$.
Finally let $1\leq i\leq k-1$.
 Let $(i(\k-1)+1,j,i(\k-1)+1+j)\in T_i$ where $\k\leq j\leq \k+i-2$.
Then $((i+1)(\k-1)+1,j+1,i(\k-1)+1+j)\in T_{i+1}'$.



An {\em intercalate} in a Latin square is a $2\times 2$ subsquare. The construction in this section shows the potential of using trades to construct MOLS with particular properties. We demonstrate this with the following theorem. 

\begin{theorem}
Let $p$ be a prime such that $p\equiv 1$ {\rm(mod $6$)}. Then there exists a Latin square $L$ orthogonal to $B_p$ such that $L$ contains an intercalate. 
\end{theorem}

\begin{proof}
Let $L:=(B_p\setminus T)\cup T'$, where $T$ and $T'$ are defined as 
in this section. We have shown above that $L$ is orthogonal to $B_p(k)$. 
Observe that  
 $(k-1,1,2k),(k-1,k,k)$ and $(k,1,k)$ are each elements of $T_1'$ and thus $L$. 
Finally, cell $(k,k)$ is not included in $T$ so $(k,k,2k)\in L$.  
\end{proof}

\section{Computational results}
In this section, we give some computational results on the spectrum of the possible sizes of orthogonal trades mentioned in the previous sections. These orthogonal trades can be found as ancillary files in \cite{CDD}.

Let $S_p$ be the set of sizes so that an orthogonal trade in $B_p$ of index $(1,k)$ exists for some $k$. For $p=5$, $S_5 = \{0, 10, 15, 20, 25\}$.

The results for $p=7$ and $p=11$ are summarised in the following lemma.

\begin{lemma} The spectrum of the sizes of orthogonal trades for $p=7$ and $p=11$ are $S_7 = \{0, 14, 18, 21, 24, 25, \dots, 49\}$ and $S_{11} = \{0, 22, 33, 36, 37, \dots, 121\}$, respectively.
\end{lemma} 
Note that an orthogonal trade in $B_7$ of size $18$ is given in Figure \ref{figg1}.

Our theoretical results only considered orthogonal trades when $p$ is prime. 
A similar question can be studied for odd values of $p$ in general. 
Here $B_p(1)$ is orthogonal to $B_p(k)$ if and only if $k\not\equiv 1\mod{p}$. 
Then the spectrum of the sizes of orthogonal trades in $B_9$ is the set $\{0, 6, 9, 12, 15, 16, 18, 19, \dots, 81\}$.

In Section \ref{section:entirerows}, we investigated the orthogonal trades in $B_p$ which are constructed by permuting entire rows. These trades preserve orthogonality with one of the $p-1$ MOLS. The possible number of rows needed to be permuted are the elements of sets $\{4, 5\}$, $\{3, 5, 6, 7\}$, $\{5, 6, 7, 8, 9, 10, 11\}$ and  $\{3, 4, 6, 7, 8, 9, 10, 11, 12, 13\}$ for $p=5, 7, 11$ and 13, respectively. 

This idea can be generalised for trades in $B_p$ which preserve orthogonality with more than one of the $p-1$ MOLS. We analyse this question for orders $p=5, 7, 11$ and 13. 

We start by considering the orthogonal trades in $B_p$ which preserve orthogonality with {\em two} other MOLS from the complete set of size $p-1$ - but only those formed by permuting entire rows. So, these orthogonal trades are formed in three MOLS of order $p$. The possible number of rows needed to be permuted are the elements of sets $\{4, 5\}$, $\{6, 7\}$, $\{5, 6, 8, 9, 10, 11\}$ and  $\{4, 6, 8, 9, 10, 11, 12, 13\}$ for $p=5, 7, 11$ and 13, respectively. 

Here, the non-trivial cases occur when the number of rows are not $p-1$ or $p$. So, we continue with only those cases.

Next, we consider the orthogonal trades in $B_p$ which preserve orthogonality with {\em three} other MOLS from the complete set of size $p-1$. The possible number of rows needed to be permuted are the sets $\{5, 9\}$ and $\{6, 11\}$ for $p=11$ and 13, respectively. 

The orthogonal trades which preserve orthogonality with four of the $p-1$ MOLS can be constructed by permuting 6 or 11 rows for $p=13$. Lastly, an orthogonal trade which preserve orthogonality with five of the $p-1$ MOLS cannot be constructed by permuting entire rows for these orders.

\end{document}